\documentclass[a4paper, 11pt]{article}
\usepackage{graphicx}
\usepackage{amssymb,amsmath,scalerel}
\usepackage{amsthm}
\usepackage{thmtools, thm-restate}
\usepackage[plainpages = false]{hyperref}
\usepackage[capitalize]{cleveref}
\usepackage[utf8]{inputenc}
\usepackage{enumerate}
\usepackage{tikz-cd}
\usepackage{todonotes}
\usepackage[backend = biber, style = numeric, giveninits]{biblatex}
\usetikzlibrary{babel}

\DeclareMathOperator{\Aut}{Aut}
\DeclareMathOperator{\coker}{coker}

\DeclareMathOperator{\End}{End}
\DeclareMathOperator{\Fix}{Fix}

\DeclareMathOperator{\Id}{Id}

\DeclareMathOperator{\Inn}{Inn}
\DeclareMathOperator{\im}{Im}

\DeclareMathOperator{\SpecR}{Spec_R}

\DeclareMathOperator{\Stab}{Stab}

\newcommand*{\eg}{e.g.\ }

\newcommand*{\fg}{finitely generated\ }
\newcommand*{\fgtf}{finitely generated torsion-free\ }

\newcommand*{\grpgen}[1]{\left\langle{#1}\right\rangle}
\newcommand*{\grppres}[2]{\langle #1 \mid #2 \rangle}

\newcommand*{\hirsch}{\mathfrak{h}}
\newcommand*{\I}{\mathcal{I}}
\newcommand*{\ie}{i.e.\ }
\newcommand*{\ihat}{\hat{\imath}}
\newcommand*{\inftynorm}[1]{|#1|_{\infty}}

\newcommand*{\inn}[1]{\tau_{#1}}
\newcommand*{\invb}[1]{\inv{\left(#1\right)}}
\newcommand*{\inv}[1]{{#1}^{-1}}

\newcommand*{\N}{\mathbb{N}}

\newcommand*{\Rconj}[1]{\sim_{#1}}
\newcommand*{\Rinf}{R_{\infty}}

\newcommand*{\Reid}{\mathcal{R}}
\newcommand*{\restr}[2]{#1|_{#2}}
\newcommand*{\restrb}[2]{(#1)|_{#2}}
\newcommand*{\size}[1]{\left| #1 \right|}
\newcommand*{\tf}{torsion-free}

\newcommand*{\Z}{\mathbb{Z}}

\renewcommand{\phi}{\varphi}

\let\originalleft\left
\let\originalright\right
\renewcommand{\left}{\mathopen{}\mathclose\bgroup\originalleft}
\renewcommand{\right}{\aftergroup\egroup\originalright}

\declaretheorem[style=definition, name = Definition, numberwithin=section]{defin}
\declaretheorem[style=definition, name = Example, sibling=defin]{example}
\declaretheorem[name = Theorem, sibling=defin]{theorem}
\declaretheorem[name = Lemma, sibling=defin]{lemma}
\declaretheorem[name = Proposition, sibling=defin]{prop}
\declaretheorem[name = Corollary, sibling=defin]{cor}
\declaretheorem[style=remark, name = Remark, numbered = no]{remark}

\declaretheorem[name = Theorem, numbered = no]{theorem*}
\declaretheorem[name = Question, sibling=defin]{quest}

\numberwithin{equation}{section}

\crefname{prop}{Proposition}{Propositions}
\crefname{quest}{Question}{Questions}
\crefname{cor}{Corollary}{Corollaries}

\newcommand\extrafootertext[1]{%
    \bgroup
    \renewcommand\thefootnote{\fnsymbol{footnote}}%
    \renewcommand\thempfootnote{\fnsymbol{mpfootnote}}%
    \footnotetext[0]{#1}%
    \egroup
}

\title{The product formula for Reidemeister numbers on nilpotent groups}
\author{Pieter Senden}
\date{}                                           

\begin{document}

\maketitle

\begin{abstract}
We study the product formula for Reidemeister numbers on \fgtf nilpotent groups in two ways.
On the one hand, we generalise the product formula to central extensions.
On the other hand, we derive general results for \fg (\tf) nilpotent groups from the product formula:
we provide a strong tool to prove that an endomorphism on a \fg nilpotent group has infinite Reidemeister number and compute Reidemeister numbers on finite index subgroups of \fgtf nilpotent groups.
Using the latter, we provide an infinite family of groups with full Reidemeister spectrum.
\end{abstract}

\extrafootertext{E-mail: \url{pieter.senden@telenet.be}, ORCID: 0000-0002-3107-6775}
\extrafootertext{2020 \emph{Mathematics Subject Classification}. Primary: 20E45, 20F18; Secondary: 20E22.}
\extrafootertext{\emph{Keywords and phrases:} twisted conjugacy, Reidemeister numbers, Reidemeister spectrum, nilpotent groups, central extensions}

\section{Introduction}
Let \(G\) be a group and \(\phi \in \End(G)\) an endomorphism.
Two elements \(x, y \in G\) are \emph{\(\phi\)-conjugate} if \(x = zy \inv{\phi(z)}\) for some \(z \in G\).
This defines an equivalence relation on \(G\), and the number of equivalence classes, \(R(\phi)\), is called the \emph{Reidemeister number} of \(\phi\).
If \(\phi\) is not specified, we also speak of \emph{twisted conjugacy}.
The set \(\SpecR(G) := \{R(\psi) \mid \psi \in \Aut(G)\}\) is called the \emph{Reidemeister spectrum} of \(G\).
If \(\SpecR(G) = \{\infty\}\), then \(G\) is said to have the \emph{\(\Rinf\)-property}.
If \(\SpecR(G) = \N_{0} \cup \{\infty\}\), then \(G\) is said to have \emph{full Reidemeister spectrum}.

Twisted conjugacy and Reidemeister numbers arise from Nielsen fixed-point theory \cite{Jiang83}.
One of the goals in this field is to determine the minimal number of fixed points of maps homotopic with a given continuous map \(f: X \to X\) on a compact manifold \(X\).
A lower bound for this minimal number is the Nielsen number \(N(f)\), in the sense that \(N(f) \leq \size{\Fix(g)}\) for all \(g\) homotopic to \(f\).
For large classes of manifolds (\eg manifolds of dimension at least \(3\) \cite{Wecken40,Wecken41,Wecken41a}), \(N(f)\) is in fact the minimum, \ie there is a map \(g\) homotopic to \(f\) with \(N(f) = \size{\Fix(g)}\).
Additionaly, for nilmanifolds, the Reidemeister number of the induced map \(f_{*}: \pi_{1}(X) \to \pi_{1}(X)\) on the fundamental group yields complete information about \(N(f)\): either \(N(f) = R(f_{*})\) if the latter is finite, or \(N(f) = 0\) if \(R(f_{*}) = \infty\) \cite{HeathKeppelmann97}.
As the fundamental group of a nilmanifold is nilpotent, this is a topological explanation why twisted conjugacy in nilpotent groups has been extensively studied.

Algebraically speaking, studying twisted conjugacy in nilpotent groups is very accessible due to the so-called product formula and its criterion about infinite Reidemeister numbers:
\begin{theorem}[{Product formula, \cite[Theorem~2.6]{Romankov11}}]	\label{theo:productFormulaFGTFNilpotentGroups}
	Let \(N\) be a \fgtf nilpotent group and let \(\phi \in \End(N)\). Suppose that
	\[
		1 = N_{0} \lhd N_{1} \lhd \ldots \lhd N_{c - 1} \lhd N_{c} = N
	\]
	is a central series of \(N\) such that
	\begin{itemize}
		\item for all \(i \in \{0, \ldots, c\}\), \(\phi(N_{i}) \leq N_{i}\),
		\item for all \(i \in \{1, \ldots, c\}\), \(N_{i} / N_{i - 1}\) is torsion-free.
	\end{itemize}
	Let \(\phi_{i}\) denote the induced endomorphism on \(\frac{N_{i}}{N_{i - 1}}\) for \(i \in \{1, \ldots, c\}\). Then
	\[
		R(\phi) = \prod_{i = 1}^{c} R(\phi_{i}).
	\]
	In particular, \(R(\phi) = \infty\) if and only if \(R(\phi_{i}) = \infty\) for some \(i \in \{1, \ldots, c\}\)
\end{theorem}

The power of the product formula comes from the fact that Reidemeister numbers on free abelian groups of finite rank are essentially given by determinants. For \(n \in \Z\), we put
\[
	\inftynorm{n} = \begin{cases}
		|n|	&	\mbox{if } n \ne 0 \\
		\infty	&	\mbox{if } n = 0.
	\end{cases}
\]
\begin{prop}[{\cite[\S~3]{Romankov11}}]	\label{prop:ReidemeisterNumbersFreeAbelianGroups}
	Let \(A\) be a \fgtf abelian group and let \(\phi \in \End(A)\). Then
	\[
		R(\phi) = \size{\coker(\phi - \Id_{A})} = [A : \im(\phi - \Id_{A})] = \inftynorm{\det(\phi - \Id_{A})}.
	\]
	In particular, \(R(\phi) = \infty\) if and only if \(\phi\) has a non-trivial fixed point.
\end{prop}

Consequently, on \fgtf nilpotent groups, calculating Reidemeister numbers boils down to calculating determinants.

The product formula and the criterion for infinite Reidemeister numbers form the key tool for various results: for free nilpotent groups \cite{GoncalvesWong09,Romankov11,DekimpeGoncalves14}, nilpotent quotients of surface groups \cite{DekimpeGoncalves16} and nilpotent quotients of Baumslag-Solitar groups \cite{DekimpeGoncalves20}, it has been classified which of them have the \(\Rinf\)-property; the complete Reidemeister spectrum is known for the free nilpotent groups of class \(2\) \cite{DekimpeTertooyVargas20} and several \(2\)-step nilpotent quotients of right-angled Artin groups \cite{DekimpeLathouwers23}.

In addition, nilpotent groups have often formed a tool to prove that non-nilpotent groups have the \(\Rinf\)-property, such as right-angled Artin groups \cite{Witdouck24} and pure Artin braid groups \cite{DekimpeGoncalvesOcampo21}.

In contrast to its intensive usage, the product formula itself and its potential for \fgtf nilpotent groups in general has been studied a lot less. We only know of  \cite{FelshtynKlopsch22}, where A.\ Fel'shtyn and B.\ Klopsch have generalised the product formula to bi-twisted conjugacy of endomorphisms on \tf{} nilpotent groups of finite Pr\"{u}fer rank. 

The goal of this paper is therefore twofold. Firstly, we want to investigate more closely which part of the product formula fully relies on nilpotency and which part can be proven under weaker conditions. This leads to the study of Reidemeister numbers in central extensions. More specifically, we obtain an addition formula (\cref{theo:sumFormulaReidemeisterNumberCentralExtension}) and a more general product formula (\cref{prop:productFormulaCentralExtension}). Secondly, we want to study the applications of the product formula in general nilpotent groups, instead of applying it to concrete families of groups. This leads to strong tools to show that an endomorphism of a \fg nilpotent group -- not necessarily torsion-free -- has infinite Reidemeister number (\cref{prop:InfiniteReidemeisterNumberOnSubgroupFGTFNilpotentGroup,theo:InfiniteReidemeisterNumberOnSubgroupFGNilpotentGroup}) and to the study of Reidemeister numbers on finite index subgroups of nilpotent groups (\cref{theo:ReidemeisterNumberEndomorphismFiniteIndexNormalSubgroupFGTFNilpotentSubgroupEqualsReidemeisterNumber}). We use the latter to establish a new infinite family of groups with full Reidemeister spectrum.

We fix some notation for this article. For a group \(G\), an endomorphism \(\phi \in \End(G)\), and a subgroup \(H\) with \(\phi(H) \leq H\), we write \(\restr{\phi}{H} \in \End(H)\) for the restriction of \(\phi\) to \(H\). If, additionally, \(H\) is also normal, we write \(\bar{\phi}\) for the induced endomorphism on \(G / H\); if there are multiple (normal) subgroups involved, we will clarify with respect to which subgroup we take the quotient. Finally, for \(g \in G\), we let \(\inn{g}\) denote the inner automorphism associated with \(g\), \ie \(\inn{g}(x) = gx\inv{g}\) for all \(x \in G\).

\section{Central extensions}
The reason to look into central extensions to generalise the product formula for \fgtf nilpotent groups is an upper bound for Reidemeister numbers in central extensions:
\begin{lemma}	\label{lem:InequalityCentralExtensions}
	Let \(G\) be a group, \(C\) a central subgroup and \(\phi \in \End(G)\) such that \(\phi(C) \leq C\). Then
	\begin{equation}	\label{eq:InequalityCentralExtensions}
		R(\phi) \leq R(\restr{\phi}{C}) R(\bar{\phi}).
	\end{equation}
\end{lemma}
In \cite[Lemma~1.1(3)]{GoncalvesWong09}, the inequality is in fact stated as an equality;
however, equality does not hold true in general.
The simplest counterexample is to consider the cyclic group \(C_{4} = \grpgen{x}\) of order \(4\), the subgroup \(C\) generated by \(x^{2}\) and the inversion map \(\phi: C_{4} \to C_{4}: x^{i} \mapsto x^{-i}\).
Straightforward computations show that each of the maps \(\phi, \restr{\phi}{C}\) and \(\bar{\phi}\) has Reidemeister number \(2\). 
Hence,
\[
	R(\phi) = 2 < 4 = R(\restr{\phi}{C}) R(\bar{\phi}).
\]
For \fgtf nilpotent groups, however, equality does hold and it forms the lion's share of the classical proof of the product formula for nilpotent groups.
We will argue that residually finiteness instead of nilpotency, together with torsion-freeness and a slight weakening of being finitely generated, suffice to show that equality holds (\cref{prop:productFormulaCentralExtension}).

In an earlier paper \cite[Theorem~2.1]{GoncalvesWong05}, D.\ Gon\c{c}alves and P.\ Wong also mention that equality holds in \eqref{eq:InequalityCentralExtensions}, albeit starting from a different formula:
\begin{theorem}	\label{theo:AdditionFormula}
	Let \(G\) be a group, \(N\) a normal subgroup, and \(\phi \in \End(G)\) such that \(\phi(N) \leq N\). Suppose that \(\Fix(\inn{gN} \circ \bar{\phi}) = 1\) for all \(g \in G\). Then
	\[
		R(\phi) = \sum_{[gN]_{\bar{\phi}} \in \Reid[\bar{\phi}]} R(\restrb{\inn{g} \circ \phi}{N}).
	\]
	
	If, moreover, \(R(\restrb{\inn{g} \circ \phi}{N})\) is independent of \(g\), then
	\[
		R(\phi) = R(\restr{\phi}{N}) R(\bar{\phi}).
	\]
\end{theorem}

They formulate it more generally, but for the purpose of this paper, we restrict ourselves to this formulation.
They remark that the condition for the second equality to hold is met for central extensions, which thus proves equality in \eqref{eq:InequalityCentralExtensions}.

There are two assumptions needed to prove the equality \eqref{eq:InequalityCentralExtensions} using \cref{theo:AdditionFormula}: the triviality of \(\Fix(\inn{gN} \circ \bar{\phi})\) for all \(g \in G\), and the independence of \(R(\restrb{\inn{g} \circ \phi}{N})\) of \(g\).
The latter comes from central extensions, the former from residual finiteness, as we will see.

Gon\c{c}alves and Wong start from the assumption on the fixed point subgroups to obtain an addition formula, then provide a broad class of extensions for which the Reidemeister numbers \(R(\restrb{\inn{g} \circ \phi}{N})\) are independent of \(g\), namely the central extensions -- which would explain their formulation in \cite[Lemma~1.1(3)]{GoncalvesWong09}.

In this article, we provide a sort of converse: we start from a central extension to obtain an addition formula, then provide a broad class of extensions and endomorphisms for which the groups \(\Fix(\inn{gN} \circ \bar{\phi})\) are trivial for all \(g\) (\cref{cor:twistedStabiliserFGTFResiduallyFinite}).
This results in a broad class of central extensions for which equality in \eqref{eq:InequalityCentralExtensions} holds (\cref{prop:productFormulaCentralExtension}).

\subsection{Addition formula for central extensions}
Our starting point is another addition formula for Reidemeister numbers. 
\begin{lemma}[{\cite[Lemma~2.1(2)]{KimLeeLee05}}]	\label{lem:GeneralAdditionFormula}
Let \(G\) be a group, \(N\) a normal subgroup, and \(\phi \in \End(G)\) such that \(\phi(N) \leq N\). For each \(g \in G\), let 
\[
	\ihat_{g}: \Reid[\restrb{\inn{g} \circ \phi}{N}]  \to \Reid[\inn{g} \circ \phi]
\]
be the natural inclusion of Reidemeister classes. Then

	\[
	R(\phi) = \sum_{[gN]_{\bar{\phi}} \in \Reid[\bar{\phi}]} \size{\im \ihat_{g}}
	\]
\end{lemma}

Throughout this section, \(G\) is a fixed group, \(C\) a central subgroup, and \(\phi \in \End(G)\) an endomorphism of \(G\) such that \(\phi(C) \leq C\).

By \cref{lem:GeneralAdditionFormula},
\[
	R(\phi) = \sum_{[gC]_{\bar{\phi}} \in \Reid[\bar{\phi}]} \size{\im \ihat_{g}},
\]
where \(\ihat_{g}: \Reid[\restr{(\inn{g} \circ \phi)}{C}] \to \Reid[\inn{g} \circ \phi]\). Now, since \(C\) is central, \(\inn{g}\) restricted to \(C\) is the identity map for each \(g \in G\). Therefore, the domain of \(\ihat_{g}\) is \(\Reid[\restr{\phi}{C}]\) for each \(g \in G\).

Let \(g \in G\) and \(z_{1}, z_{2} \in C\) be arbitrary. We have the following chain of equivalences:
\begin{align*}
	\ihat_{g}([z_{1}]_{\restr{\phi}{C}}) = \ihat_{g}([z_{2}]_{\restr{\phi}{C}})	&\iff \exists x \in G: z_{1} = xz_{2} g \inv{\phi(x)} \inv{g}	\\
	&\iff \exists x \in G: z_{1} = \inv{g} x g \inv{\phi(x)} z_{2},
\end{align*}
where we conjugate by \(g\) in the last equivalence and simultaneously use that \(C\) is central. Furthermore, if we project the last equality down to \(G / C\), we find \(C = \inv{g}xg\inv{\phi(x)}C\), or, equivalently, \(gC = xg \inv{\phi(x)}C\). Thus, if we write
\[
	\Stab_{\bar{\phi}}(gC) := \{hC \in G / C \mid gC = hg\inv{\phi(h)}C\},
\]
then
\[
\begin{split}
	&\exists x \in G: z_{1} = z_{2} \inv{g} x g \inv{\phi(x)} \\ &\iff \exists x \in \inv{\pi}(\Stab_{\bar{\phi}}(gC)): z_{1} = \inv{g} x g \inv{\phi(x)} z_{2},
\end{split}
\]
where \(\pi: G \to G / C\) is the canonical projection. Consequently, if we define the \emph{\(\phi\)-twisted commutator} \(g, x \in G\) as \([g, x]^\phi := \inv{g} \inv{x} g \phi(x)\), we see that
\[
	\ihat_{g}([z_{1}]_{\restr{\phi}{C}}) = \ihat_{g}([z_{2}]_{\restr{\phi}{C}}) \iff z_{1} = [g, x]^\phi z_{2} \text{ for some } x \in \inv{\pi}(\Stab_{\bar{\phi}}(gC)).
\]

In general, for any endomorphism \(\psi\) of \(G\) and \(g \in G\), the set
\[
	\Stab_{\psi}(g) = \{x \in G \mid xg\inv{\psi(x)} = g\}
\]
is called the \emph{\(\psi\)-twisted stabiliser} of \(g\). It is a subgroup, as it is the stabiliser of \(g\) for the group action
\[
	G \times G \to G: (x, g) \mapsto xg\inv{\psi(x)}.
\]

\begin{lemma}	\label{lem:twistedCommutatorSubgroup}
	For each \(g \in G\), the set
	\[
		[g, \inv{\pi}(\Stab_{\bar{\phi}}(gC))]^\phi := \{[g, x]^\phi \mid x \in \inv{\pi}(\Stab_{\bar{\phi}}(gC))\}
	\]
	is a subgroup of \(C\).
\end{lemma}
\begin{proof}
	First, if \(x \in \inv{\pi}(\Stab_{\bar{\phi}}(gC))\), then \(gC = \inv{x}g\phi(x)C\) and thus \(C = \inv{g}\inv{x} g \phi(x)C = [g, x]^\phi C\), which shows that \([g, x]^\phi \in C \leq Z(G)\). It follows that \(h[g, x]^\phi \inv{h} = [g, x]^\phi\) for all \(h \in G\) and \(x \in \inv{\pi}(\Stab_{\bar{\phi}}(gC))\).
	Consequently,
		\begin{align*}
			\invb{[g, x]^\phi}	&= \phi(x)\invb{[g, x]^\phi} \inv{\phi(x)}				\\
								&= \phi(x) \inv{\phi(x)} \inv{g} x g \inv{\phi(x)}	\\
								&= \inv{g} x g \inv{\phi(x)}								\\
								&= \left[g, \inv{x}\right]^\phi \in [g, \inv{\pi}(\Stab_{\bar{\phi}}(gC))]^\phi
		\end{align*}
	
	Next, let \(x_1, x_2 \in \inv{\pi}(\Stab_{\bar{\phi}}(gC))\). Then 
	\begin{align*}
		[g, x_1]^\phi [g, x_2]^\phi	&=	\inv{g} \inv{x}_1 g \phi(x_1) \inv{g} \inv{x}_2g \phi(x_2)	\\
									&= \inv{g} \left(\inv{x}_1 g \phi(x_1) \inv{g}\right) \inv{x}_2 g \phi(x_2) \\
									&= \inv{g} \inv{x}_2 \left(\inv{x}_1 g \phi(x_1) \inv{g}\right) g \phi(x_2) \\
									&= \inv{g} \inv{x}_2 \inv{x}_1 g \phi(x_1) \phi(x_2)		\\
									&= \inv{g} \invb{x_1 x_2} g \phi(x_1 x_2)				\\
									&= [g, x_1 x_2]^\phi \in [g, \inv{\pi}(\Stab_{\bar{\phi}}(gC))]^\phi
	\end{align*}
	since \(\inv{x}_1 g \phi(x_1) \inv{g} = g\left([g, x_1]^\phi\right) \inv{g}\in C \leq Z(G)\). Therefore, the set \([g, \inv{\pi}(\Stab_{\bar{\phi}}(gC))]^\phi\) is a subgroup.
\end{proof}

\begin{prop}
	For each \(g \in G\), \(\size{\im \ihat_{g}}\) equals the index of \([g, \inv{\pi}(\Stab_{\bar{\phi}}(gC))]^\phi\) in \(C\).
\end{prop}
\begin{proof}
We have argued earlier that, for \(g \in G\) and \(z_{1}, z_{2} \in C\), \(\ihat_{g}([z_{1}]_{\restr{\phi}{C}}) = \ihat_{g}([z_{2}]_{\restr{\phi}{C}})\) if and only if \(z_{1} = x z_{2}\) for some \(x \in [g, \inv{\pi}(\Stab_{\bar{\phi}}(gC))]^{\phi}\). Thus, the map
\[
	C / [g, \inv{\pi}(\Stab_{\bar{\phi}}(gC))]^{\phi} \to \Reid[\inn{g} \circ \phi]: z[g, \inv{\pi}(\Stab_{\bar{\phi}}(gC))]^{\phi} \to \ihat_{g}([z]_{\restr{\phi}{C}})
\]
is injective and its image is equal to \(\im \ihat_{g}\). This yields the result.
\end{proof}

Combining this with \cref{lem:GeneralAdditionFormula}, we immediately get the following:

\begin{theorem}	\label{theo:sumFormulaReidemeisterNumberCentralExtension}
	With the notation as above,
	\[
		R(\phi) = \sum_{[gC]_{\bar{\phi}} \in \Reid[\bar{\phi}]} \Big[C : [g, \inv{\pi}(\Stab_{\bar{\phi}}(gC))]^{\phi} \Big]
	\]
\end{theorem}

For central extensions, \cref{lem:GeneralAdditionFormula} thus involves computing indices of subgroups rather than sizes of mere subsets. In this form, however, it is rather unclear how to derive the inequality \(R(\phi) \leq R(\restr{\phi}{C}) R(\bar{\phi})\). We address this issue by investigating the expression 
\[
	\Big[C : [g, \inv{\pi}(\Stab_{\bar{\phi}}(gC))]^{\phi} \Big].
\]

We start by noting that \(C\) is a central subgroup, so it is abelian. In abelian groups, the set of Reidemeister classes carries a natural group structure \cite[\S 2]{Jiang83}, \cite[Proposition~1.3]{Heath85}, \cite[\S 3]{Romankov11}:
\begin{prop}	\label{prop:GroupStructureReidemeisterClassesAbelianGroups}
Let \(A\) be an abelian group and \(\psi \in \End(A)\). Then the set of Reidemeister classes is given by the cosets of \(\im(\psi - \Id_{A})\), where \([a]_{\psi} = a + \im(\psi - \Id_{A})\). 
\end{prop}

In our case, the set \(\Reid[\restr{\phi}{C}]\) thus consists of the cosets of \(I := \{c \inv{\phi(c)} \mid c \in C\}\), \ie \(\Reid[\restr{\phi}{C}] = C / I\). Moreover, for each \(g \in G\), the group \(I\) is a subgroup of \([g, \inv{\pi}(\Stab_{\bar{\phi}}(gC))]^\phi\), as \(c \inv{\phi(c)} = [g, c]^{\phi}\) for all \(c \in C, g \in G\). This implies that
\[
	\Big[C : [g, \inv{\pi}(\Stab_{\bar{\phi}}(gC))]^{\phi} \Big] = \left[ \frac{C}{I} : \frac{[g, \inv{\pi}(\Stab_{\bar{\phi}}(gC))]^{\phi}}{I} \right].
\]
Since \(R(\restr{\phi}{C}) = [C :  I]\), \cref{theo:sumFormulaReidemeisterNumberCentralExtension} then yields the desired inequality:
\begin{align*}
	R(\phi)	&= \sum_{[gC]_{\bar{\phi}} \in \Reid[\bar{\phi}]} \Big[C : [g, \inv{\pi}(\Stab_{\bar{\phi}}(gC))]^{\phi} \Big]	\\
			&\leq \sum_{[gC]_{\bar{\phi}} \in \Reid[\bar{\phi}]} R(\restr{\phi}{C})	\\
			&= R(\restr{\phi}{C}) R(\bar{\phi}).
\end{align*}

\begin{example}
	We apply \cref{theo:sumFormulaReidemeisterNumberCentralExtension} to the example on \(G := C_{4} = \grpgen{x}\) we gave earlier. Recall that we took \(C\) equal to the subgroup generated by \(x^{2}\) and \(\phi\) equal to the inversion map. Since \(\bar{\phi}\) is the identity (the only automorphism on the cyclic group of order \(2\)), \(\Stab_{\bar{\phi}}(gC) = G / C\) for all \(g \in G\). Next, for \(g \in G\),
	\begin{align*}
		[g, \inv{\pi}(\Stab_{\bar{\phi}}(gC))]^{\phi}	&= [g, G]^{\phi}	\\
										&= \{\inv{g}\invb{x^{i}} g \phi(x^{i}) \mid i \in \Z\} \\
										&= \{x^{-i} x^{-i} \mid i \in \Z\} \\
										&= \{1, x^{2}\} \\
										&= C,
	\end{align*}
	as \(G\) is abelian. Consequently, \cref{theo:sumFormulaReidemeisterNumberCentralExtension} yields
	\[
		R(\phi) = \sum_{[gC]_{\bar{\phi}} \in \Reid[\bar{\phi}]} [C : C] = R(\bar{\phi}) = 2.
	\]
\end{example}
\subsection{Product formula for central extensions of residually finite groups}

Theoretically, the formula in \cref{theo:sumFormulaReidemeisterNumberCentralExtension} can be used to compute Reidemeister numbers, but doing so is generally work-intensive. However, for a large family of groups, this formula simplifies to \(R(\phi) = R(\restr{\phi}{C}) R(\bar{\phi})\), the equality-part from \cref{lem:InequalityCentralExtensions}.

\begin{defin}
	A group \(G\) is of \emph{type (F)} if, for all positive integers \(n\), \(G\) only has finitely many subgroups of index \(n\).
\end{defin}
\begin{prop}	\label{prop:productFormulaCentralExtension}
	Let \(G\) be a group and \(C\) a central subgroup such that \(G / C\) is \tf{} residually finite and of type \((F)\). Let \(\phi \in \End(G)\) be such that \(\phi(C) \leq C\). Let \(\restr{\phi}{C}\) and \(\bar{\phi}\) denote the induced endomorphisms on \(C\) and \(G / C\), respectively. Then
	\[
		R(\phi) = R(\restr{\phi}{C})R(\bar{\phi}).
	\]
\end{prop}
Finitely generated groups are of type \((F)\) \cite{Hall49}.
Topologically finitely generated profinite groups are of type \((F)\) as well, by combining \cite[Theorem~2]{Anderson76} and \cite[Theorem~1.1]{NikolovSegal03}.
Note that these include groups that are not finitely generated in the algebraic sense, such as the \(p\)-adic integers.
In addition, profinite groups are also residually finite, and the quotient of a profinite group by a closed subgroups is also profinite. 
Hence, \cref{prop:productFormulaCentralExtension} can be applied to topologically finitely generated profinite groups.

The result hinges on the following:

\begin{prop}	\label{prop:infiniteTwistedStabilisersImpliesInfiniteReidemeisterNumber}
	Let \(G\) be a residually finite group of type \((F)\). Let \(\phi \in \End(G)\). If \(\Stab_{\phi}(g)\) is infinite for some \(g \in G\), then \(R(\phi) = \infty\).
\end{prop}

As far as we know, E.\ Jabara is the first to use this result \cite{Jabara08}, albeit for automorphisms on finitely generated residually finite groups and without explicitly stating it in this form.
In \cite[Proposition~3.7]{Senden21}, we explicitly prove the result Jabara uses.
A similar proof also works for the version in \cref{prop:infiniteTwistedStabilisersImpliesInfiniteReidemeisterNumber}, and for the reader's convenience, we will present it here.

The key ingredient of the proof is an upper bound for Reidemeister numbers on finite groups. 
Jabara proves this for automorphisms \cite[Lemma~4]{Jabara08}, but his proof also works for endomorphisms:
\begin{lemma}	\label{lem:SizeFixpointsBoundedByReidemeisterNumber}
	Let \(G\) be a finite group and \(\phi \in \End(G)\). Then \(\size{\Fix(\phi)} \leq 2^{2^{R(\phi)}}\)
\end{lemma}

To make the proof in \cite[Proposition~3.7]{Senden21} work for endomorphisms, we need several lemmata. The first one is a slight generalisation of one by T.\ Hsu and D.\ Wise \cite[Lemma~2.2]{HsuWise03}.
\begin{lemma}	\label{lem:FullyCharacteristicFiniteIndexSubgroup}
	Let \(G\) be a group of type \((F)\) and \(H\) a proper subgroup of finite index. Then there exists a fully characteristic subgroup \(K\) of finite index such that \(K\) is contained in \(H\).
\end{lemma}
\begin{proof}
	Define \(X := \{L \leq G \mid [G : L] \leq [G : H]\}\) and
	\[
		K := \bigcap_{L \in X} L.
	\]
	Since \([G : H]\) is finite and \(G\) is of type \((F)\), \(X\) is a finite set.
	Therefore, \(K\) has finite index, as an intersection of finitely many subgroups of finite index.
	Also, \(H \in X\), so \(K \leq H\).
	We are only left to argue that \(K\) is fully characteristic.
	Let \(\phi \in \End(G)\) and let \(L \in X\).
	Note that \(\ker \phi \leq \inv{\phi}(L)\).
	Consequently,
	\begin{align*}
		[G : \inv{\phi}(L)]	&= \left[ \frac{G}{\ker \phi} : \frac{\inv{\phi}(L)}{\ker \phi}\right]	\\
						&= \left [\frac{G}{\ker \phi} : \frac{\inv{\phi}(\im \phi \cap L)}{\ker \phi} \right]	\\
						&= [\im \phi : \im \phi \cap L]	\\
						&\leq [G : L].
	\end{align*}
	Therefore, \(\inv{\phi}(L) \in X\).
	It follows that
	\[
		\inv{\phi}(K) = \bigcap_{L \in X} \inv{\phi}(L) \geq \bigcap_{L \in X} L = K,
	\]
	which implies that \(\phi(K) \leq \phi(\inv{\phi}(K)) \leq K\).
	Therefore, \(K\) is fully characteristic, since \(\phi\) is arbitrary.
\end{proof}

The next two lemmata are well-known results in twisted conjugacy.
\begin{lemma}[{\cite[Lemma~1.1(1)]{GoncalvesWong09}}]	\label{lem:LowerBoundReidemeisterNumberByQuotient}
	Let \(G\) be a group, \(N\) a normal subgroup, and \(\phi \in \End(G)\) such that \(\phi(N) \leq N\). Let \(\bar{\phi}\) denote the induced endomorphism on \(G / N\). Then \(R(\phi) \geq R(\bar{\phi})\).
\end{lemma}

\begin{lemma}	\label{lem:PropertiesTwistedConjugacyCompositionWithInnerAutomorphism}
	Let \(G\) be a group, \(\phi \in \End(G)\), and \(g \in G\). Then
	\begin{enumerate}[(1)]
		\item \(R(\phi) = R(\inn{g} \circ \phi)\),
		\item \(\Stab_{\phi}(g) = \Fix(\inn{g} \circ \phi)\).
	\end{enumerate}
\end{lemma}
\begin{proof}
	For the first item, we refer the reader to \cite[Corollary~2.5]{FelshtynTroitsky07}.
	
	For the second, the computation is straightforward:
	\begin{align*}
		\Stab_{\phi}(g)	&=	\{x \in G \mid xg \inv{\phi(x)} = g \}	\\
						&=	\{x \in G \mid x = g\phi(x) \inv{g}\}	\\
						&=	\Fix(\inn{g} \circ \phi).\qedhere
	\end{align*}
\end{proof}

\begin{proof}[Proof of \cref{prop:infiniteTwistedStabilisersImpliesInfiniteReidemeisterNumber}]
	Since \(\Stab_{\phi}(g) = \Fix(\inn{g} \circ \phi)\) and \(R(\inn{g} \circ \phi) = R(\phi)\) by the previous lemma, we may switch to \(\inn{g} \circ \phi\) and thus assume that \(\Fix(\phi)\) is infinite.
	
	Fix \(n \geq 1\) and let \(g_{1}, \ldots, g_{n}\) be (distinct) elements in \(\Fix(\phi)\). Since \(G\) is residually finite, we can find a finite index normal subgroup \(N\) such that \(x_{1}N, \ldots, x_{n}N\) are all distinct. Since \(G\) is of type \((F)\), we can find, using \cref{lem:FullyCharacteristicFiniteIndexSubgroup}, a fully characteristic subgroup \(K\) of finite index contained in \(N\) such that \(x_{1}K, \ldots, x_{n}K\) are all distinct. Let \(\bar{\phi}\) denote the induced endomorphism on \(G / K\). Then \(\{x_{1}K, \ldots, x_{n}K\} \subseteq \Fix(\bar{\phi})\), which implies that
	\[
		\size{\Fix(\bar{\phi})} \geq n.
	\]
	By \cref{lem:SizeFixpointsBoundedByReidemeisterNumber}, \(\size{\Fix(\bar{\phi})} \leq 2^{2^{R(\bar{\phi})}}\), which implies that
	\[
		R(\bar{\phi}) \geq \log_{2}(\log_{2}(\size{\Fix(\bar{\phi})})) \geq \log_{2} (\log_{2}(n)).
	\]
	By \cref{lem:LowerBoundReidemeisterNumberByQuotient}, \(R(\phi) \geq R(\bar{\phi})\). Combining everything, we get
	\[
		R(\phi) \geq \log_{2} (\log_{2}(n)).
	\]
	As \(n\) was arbitrary and \(\log_{2}(\log_{2}(n))\) tends to infinity as \(n\) does, we obtain the equality \(R(\phi) = \infty\).

\end{proof}

\begin{cor}	\label{cor:twistedStabiliserFGTFResiduallyFinite}
	Let \(G\) be a torsion-free residually finite group of type \((F)\) and let \(\phi \in \End(G)\). If \(R(\phi) < \infty\), then \(\Stab_{\phi}(g)\) is trivial for all \(g \in G\).
\end{cor}
\begin{proof}
	\cref{prop:infiniteTwistedStabilisersImpliesInfiniteReidemeisterNumber} implies that \(\Stab_{\phi}(g)\) is finite for all \(g \in G\). As twisted stabilisers are subgroups, they have to be trivial, since \(G\) is torsion-free.
\end{proof}
The condition that \(G\) is of type \((F)\) cannot be dropped, see \cite[Example~3.9]{Senden21}. With this corollary, we can prove \cref{prop:productFormulaCentralExtension}.

\begin{proof}[Proof of \cref{prop:productFormulaCentralExtension}]
	If \(R(\bar{\phi}) = \infty\), then \(R(\phi) = \infty\) as well by \cref{lem:LowerBoundReidemeisterNumberByQuotient}, so the equality \(R(\phi) = R(\restr{\phi}{C}) R(\bar{\phi})\) holds in that case. Therefore, assume that \(R(\bar{\phi}) < \infty\).  \cref{cor:twistedStabiliserFGTFResiduallyFinite} then implies that the twisted stabilisers of \(\bar{\phi}\) are trivial. Consequently,
	\[
		\inv{\pi}(\Stab_{\bar{\phi}}(gC)) = \inv{\pi}(1) = C
	\]
	for all \(g \in G\). This implies that
	\[
		[g, \inv{\pi}(\Stab_{\bar{\phi}}(gC))]^{\phi}  = [g, C]^{\phi} = \{\inv{g} \inv{c} g \phi(c) \mid c \in C \} = \{\inv{c} \phi(c) \mid c \in C\} = I 
	\]
	for all \(g \in G\), as \(C \leq Z(G)\). Recall that \(R(\restr{\phi}{C}) = [C : I]\). Therefore, using \cref{theo:sumFormulaReidemeisterNumberCentralExtension}, we obtain
	\[
		R(\phi) = \sum_{[gC]_{\bar{\phi}} \in \Reid[\bar{\phi}]} [C : I] = R(\restr{\phi}{C})R(\bar{\phi}).	\qedhere
	\]

\end{proof}

As mentioned earlier, we can also derive \cref{prop:productFormulaCentralExtension} from \cref{theo:AdditionFormula} if \(R(\bar{\phi}) < \infty\), since in that case, the twisted stabiliser of \(\bar{\phi}\) are trivial. \cref{theo:AdditionFormula} then reads
	\[
		R(\phi) = \sum_{[gC]_{\bar{\phi}} \in \Reid[\bar{\phi}]} R(\restr{(\inn{g} \circ \phi)}{C}).
	\]
	Since \(C\) is central, \(\inn{g}\) restricted to \(C\) is the identity map for each \(g \in G\). Consequently, the summation above simplifies to
	\[
		R(\phi) = \sum_{[gC]_{\bar{\phi}} \in \Reid[\bar{\phi}]} R(\restr{\phi}{C}) = R(\restr{\phi}{C})R(\bar{\phi}). \qedhere
	\]

\begin{cor}	\label{cor:productFormulaCentreExtensionResiduallyFinite}
	Let \(G\) be a residually finite group of type \((F)\) such that \(G / Z(G)\) is torsion-free. Let \(\phi \in \Aut(G)\). Write \(\restr{\phi}{Z(G)}\) and \(\bar{\phi}\) for the induced automorphisms on \(Z(G)\) and \(G / Z(G)\), respectively. Then \(R(\phi) = R(\restr{\phi}{Z(G)})R(\bar{\phi})\).
\end{cor}
\begin{proof}
	Since \(G\) is residually finite of type \((F)\), \(\Aut(G)\) is residually finite as well, a famous result due to G.\ Baumslag \cite[Theorem~1]{Baumslag63}; he proves it for finitely generated groups, but only uses the fact that they are finitely generated to prove they are (in our terminology) of type \((F)\).
	Hence, his result also holds for groups of type \((F)\).
	Next, as subgroups of residually finite groups are residually finite, \(G / Z(G) \cong \Inn(G) \leq \Aut(G)\) is residually finite.
	By assumption, \(G / Z(G)\) is torsion-free.
	Finally, as there is a one-to-one correspondence between the subgroups of \(G / Z(G)\) and those of \(G\) containing \(Z(G)\), \(G / Z(G)\) is also of type \((F)\).
	Therefore, \cref{prop:productFormulaCentralExtension} applies and we derive \(R(\phi) = R(\restr{\phi}{Z(G)})R(\bar{\phi})\).
\end{proof}

\section{Three applications of the product formula}
From this section onwards, we focus on \fgtf nilpotent groups.
We start with proving the product formula for nilpotent groups (\cref{theo:productFormulaFGTFNilpotentGroups}) using the one for central extensions of residually finite groups.
Afterwards, we discuss three applications of the product formula.

\begin{proof}[Proof of \cref{theo:productFormulaFGTFNilpotentGroups}]
	We proceed by induction on \(c\). If \(c = 1\), then \(\phi_{1}\) is essentially \(\phi\), in which case the product formula clearly holds. So, suppose it holds for all \fgtf nilpotent groups and central series of length \(c - 1\) as in the statement. Let
	\[
		1 = N_{0} \lhd N_{1} \lhd \ldots \lhd N_{c - 1} \lhd N_{c} = N
	\]
	be one of length \(c\). Since \(N_{1}\) is a central subgroup and \(N / N_{1}\) is \fgtf nilpotent, hence residually finite and of type \((F)\), \cref{prop:productFormulaCentralExtension} implies that
	\[
		R(\phi) = R(\phi_{1}) R(\bar{\phi}),
	\]
	where \(\bar{\phi}\) is the induced endomorphism on \(N / N_{1}\). The induction hypothesis yields
	\[
		R(\bar{\phi}) = \prod_{i = 2}^{c} R(\bar{\phi}_{i}),
	\]
	where \(\bar{\phi}_{i} \in \End \left( \frac{N_{i}/N_{1}}{N_{i - 1}/N_{1}}\right)\) for \(i \in\{2, \ldots c\}\). Indeed, it is straightforward that
	\[
		1 = \frac{N_{1}}{N_{1}} \lhd \frac{N_{2}}{N_{1}} \lhd \ldots \lhd \frac{N_{c - 1}}{N_{1}} \lhd \frac{N_{c}}{N_{1}}
	\]
	is a central series of length \(c - 1\) that satisfies the conditions from the statement of the product formula.
	Also, the diagram
	\[
		\begin{tikzcd}
			& N_{i} / N_{i - 1}	\ar[r, "\cong"] \ar[d, "\phi_{i}"]	& \frac{N_{i} / N_{1}}{N_{i - 1} / N_{1}}		\ar[d, "\bar{\phi}_{i}"]	\\
			& N_{i} / N_{i - 1}	\ar[r, "\cong"]				& \frac{N_{i} / N_{1}}{N_{i - 1} / N_{1}}
		\end{tikzcd}
	\]
	commutes and thus shows that \(R(\phi_{i}) = R(\bar{\phi}_{i})\). Consequently,
	\[
		R(\phi) = R(\phi_{1}) R(\bar{\phi}) = R(\phi_{1})\prod_{i = 2}^{c} R(\bar{\phi}_{i}) = \prod_{i = 1}^{c} R(\phi_{i}). \qedhere
	\]
\end{proof}

For ease of reference, we call a central series that satisfies the two additional conditions from \cref{theo:productFormulaFGTFNilpotentGroups} a \emph{\(\phi\)-invariant \tf{} central series}, or \(\phi\)-ITF-series for short.

\subsection{Application 1: a sufficient condition for infinite Reidemeister number}
In the literature, there is one central series that is ubiquitous in computations involving the product formula: the \emph{adapted lower central series}.
Recall that the lower central series of a group \(G\) is given by \(\gamma_{1}(G) := G\) and \(\gamma_{i + 1}(G) := [\gamma_{i}(G), G]\) for \(i \geq 1\).

To define the adapted lower central series, we need the notion of isolators.

\begin{defin}
Let \(G\) be a group and \(H\) a subset of \(G\). We define the \emph{isolator} of \(H\) as
\[
	\sqrt[G]{H} := \{g \in G \mid \exists k \in \Z_{> 0}: g^{k} \in H\}.
\]
\end{defin}
In general, the isolator of a subgroup is not necessarily a subgroup. For example, if \(H\) is the trivial subgroup, then \(\sqrt[G]{H}\) is the \emph{set} of all torsion elements in \(G\).

\begin{defin}
	Let \(G\) be a group. The \emph{adapted lower central series} of \(G\) given by
	\[
		\sqrt[G]{\gamma_{1}(G)} \geq \sqrt[G]{\gamma_{2}(G)} \geq \dots
	\]
\end{defin}

Whereas the lower central series characterises nilpotent groups, the adapted lower central series characterises \emph{torsion-free} nilpotent groups.
\begin{lemma}	\label{lem:TFNilpotentIffALCSTerminates}
	Let \(N\) be a group. Then \(N\) is torsion-free nilpotent if and only if the adapted lower central series of \(N\) reaches the trivial group in finitely many steps.
\end{lemma}
\begin{proof}
	On the one hand, the group \(N\) is nilpotent if and only if \(\gamma_{i}(N) = 1\) for some \(i \geq 1\). 
	On the other hand, \(N\) is torsion-free if and only if \(\sqrt[N]{1} = 1\).
	Consequently, \(N\) is torsion-free nilpotent if and only if \(\sqrt[N]{\gamma_{i}(N)} = 1\) for some \(i \geq 1\), as this encompasses both \(\gamma_{i}(N) = 1\) and \(\sqrt[N]{1} = 1\).
\end{proof}

The adapted lower central series indeed forms a \(\phi\)-ITF-series for all \(\phi \in \End(N)\), by the following lemma:
\begin{lemma}[{\cite[Lemma~1.1.2]{Dekimpe96}}]	\label{lem:PropertiesALCS}
	Let \(G\) be a group and \(\gamma_{k}(G)\) its lower central series.
	\begin{enumerate}[(1)]
		\item For all \(k \in \N_{0}\), \(\sqrt[G]{\gamma_{k}(G)}\) is a fully characteristic subgroup of \(G\).	\label{item:ALCSFullyCharacteristic}
		\item For all \(k \in \N_{0}\), \(G / \sqrt[G]{\gamma_{k}(G)}\) is torsion-free.	\label{item:ALCSFactorsTorsionfree}
		\item For all \(k, l \in \N_{0}\), \(\left[\sqrt[G]{\gamma_{k}(G)}, \sqrt[G]{\gamma_{l}(G)}\right] \leq \sqrt[G]{\gamma_{k + l}(G)}\).
	\end{enumerate}
\end{lemma}

This explains why the adapted lower central series is frequently used in computations of Reidemeister numbers of \fgtf nilpotent groups.

Using the adapted lower central series, we can construct other \(\phi\)-ITF-series based on a given normal subgroup.
We first need two properties of isolators: one with respect to endomorphisms, and one concerning nilpotent groups.

\begin{lemma}	\label{lem:IsolatorOfPhiInvariantIsPhiInvariant}
	Let \(G\) be a group, \(H\) a subgroup, and \(\phi \in \End(G)\). Suppose that \(\phi(H) \leq H\). Then also \(\phi(\sqrt[G]{H}) \subseteq \sqrt[G]{H}\).
\end{lemma}
\begin{proof}
	Let \(g \in \sqrt[G]{H}\) be arbitrary. Then \(g^{k}\in H\) for some \(k \geq 1\). Hence, 
	\[
		\phi(g)^{k} = \phi(g^{k}) \in H
	\]
	by assumption on \(H\). By definition of the isolator of \(H\), it follows that \(\phi(g) \in \sqrt[G]{H}\).
\end{proof}

\begin{lemma}[{\cite[Lemma~2.8]{Baumslag71a}}]	\label{lem:ConditionFiniteIndexSubgroupFGNilpotentGroups}
	Let \(N\) be a \fg nilpotent group and \(H\) a subgroup.
	If for every \(n \in N\), there is a \(k \in \Z_{> 0}\) such that \(n^{k} \in H\), then \(H\) has finite index in \(N\).
\end{lemma}

\begin{prop}	\label{prop:IsolatorOfSubgroupIsSubgroupNilpotentGroups}
	Let \(N\) be a nilpotent group and \(H\) a subgroup. Then \(\sqrt[N]{H}\) is a subgroup as well.
	
	If \(N\) is, in addition, finitely generated, then \(H\) has finite index in \(\sqrt[N]{H}\).
\end{prop}
\begin{proof}
	For the first claim, we refer the reader to \cite[2.5.8~Theorem]{Khukhro93}.
	
	The second follows from the previous lemma.
\end{proof}

\begin{prop}	\label{prop:NormalSubgroupInFGTFNilpotentGroupContainedInCentralSeries}
	Let \(N\) be a \fgtf nilpotent group. Let \(H\) be a normal subgroup such that \(N / H\) is torsion-free and let \(\phi \in \End(N)\) be such that \(\phi(H) \leq H\). Then there exists a \(\phi\)-ITF-series containing \(H\) as one of its subgroups.
	\end{prop}
\begin{proof}
	Since \(N\) is torsion-free nilpotent, there is a \(c\) such that \(\sqrt[N]{\gamma_{c + 1}(N)} = 1\) by \cref{lem:TFNilpotentIffALCSTerminates}. Let \(j \geq 0\) be the smallest index such that \(H \leq \sqrt[N]{\gamma_{c - j + 1}(N)}\). For \(i \in \{0, \ldots, j\}\), put
	\[
		N_{i} := \sqrt[N]{\gamma_{c - i + 1}(N)} \cap H.
	\]
	Next, let \(p: N \to N / H\) be the natural projection. Since \(N / H\) is torsion-free nilpotent, there is a \(d\) such that \(\sqrt[N / H]{\gamma_{d + 1}(N / H)} = 1\), again by \cref{lem:TFNilpotentIffALCSTerminates}. For \(i \in \{0, \ldots, d\}\), put
	\[
		N_{i + j} := \inv{p} \left(\sqrt[N / H]{\gamma_{d - i + 1}(N/H)}\right).
	\]
	We prove that the \(N_{i}\) form the desired central series. Note that we have defined \(N_{j}\) in two ways. However, they coincide, as
	\[
		H \cap \sqrt[N]{\gamma_{c - j + 1}(N)} = H
	\]
	by definition of \(j\) and
	\[
		\inv{p} \left(\sqrt[N / H]{\gamma_{d - 0 + 1}(N/H)}\right) = \inv{p}(1) = H
	\]
	as well.
	
	Put \(k := j + d\). As the adapted lower central series is descending, the \(N_{i}\) are ascending. To prove that they are central (and normal), we compute \([N_{i}, N]\) for each \(i \in \{0, \ldots, k\}\). If \(i \leq j\), then \([N_{i}, N] \leq [H, N] \leq H\), as \(H\) is normal, and
	\[
		[N_{i}, N] \leq \left[\sqrt[N]{\gamma_{c - i + 1}(N)}, N\right] \leq \sqrt[N]{\gamma_{c - i + 2}(N)}.
	\]
	Thus, \([N_{i}, N] \leq N_{i - 1}\). If \(i \geq j + 1\), then write \(i = j + l\) for some \(l \in \{1, \ldots, d\}\). With this notation, we find
	\begin{align*}
		[N_{i}, N] = [N_{j + l}, N]	&=	\left[\inv{p} \left(\sqrt[N / H]{\gamma_{d - l + 1}(N/H)}\right), N \right]	\\
					&=	\left[\inv{p} \left(\sqrt[N / H]{\gamma_{d - l + 1}(N/H)}\right), \inv{p}(N/H)\right]	\\
					&=	\inv{p}\left(\left[\sqrt[N / H]{\gamma_{d - l + 1}(N/H)}, N / H\right]\right)	\\
					&\leq \inv{p} \left(\sqrt[N / H]{\gamma_{d - l + 2}(N/H)}\right)	\\
					&= N_{j + l - 1} = N_{i - 1}.
	\end{align*}
	
	Next, we argue that the factors \(N_{i} / N_{i - 1}\) are torsion-free for all \(i \in \{1, \ldots, k\}\). Let \(i \in \{1, \ldots, k\}\) and suppose that \(i \leq j\). Then
	\[
		\frac{N_{i}}{N_{i - 1}} = \frac{N_{i}}{\sqrt[N]{\gamma_{c - i}(N)} \cap N_{i}} \cong \frac{N_{i}\sqrt[N]{\gamma_{c - i}(N)}}{\sqrt[N]{\gamma_{c - i}(N)}} \leq \frac{N}{\sqrt[N]{\gamma_{c - i}(N)}},
	\]
	where the latter is torsion-free by \cref{lem:PropertiesALCS}\eqref{item:ALCSFactorsTorsionfree}. Now, suppose that \(i \geq j + 1\) and write \(i = j + l\) for some \(l \in \{1, \ldots, d\}\). Then by the third isomorphism theorem,
	\[
		\frac{N}{N_{i - 1}} \cong \frac{N / H}{\sqrt[N / H]{\gamma_{d - l + 2}(N / H)}}. 
	\]
	The latter group is torsion-free, again by \cref{lem:PropertiesALCS}\eqref{item:ALCSFactorsTorsionfree}, and as \(N_{i} / N_{i - 1}\) is a subgroup of \(N / N_{i - 1}\), it is torsion-free as well.
	
	Finally, we prove that \(\phi(N_{i}) \leq N_{i}\) for each \(i \in \{0, \ldots, k\}\). Let \(i \in \{0, \ldots, k\}\) be arbitrary. If \(i \leq j\), then
	\begin{align*}
		\phi(N_{i})	&= \phi\left(\sqrt[N]{\gamma_{c - i + 1}(N)} \cap H\right)	\\
					&\leq \phi\left(\sqrt[N]{\gamma_{c - i + 1}(N)}\right) \cap \phi(H) \\
					&\leq \sqrt[N]{\gamma_{c - i + 1}(N)} \cap H	\\
					&= N_{i},
	\end{align*}
	by assumption on \(H\) and by \cref{lem:PropertiesALCS}\eqref{item:ALCSFullyCharacteristic}.
	
	If \(i \geq j + 1\), write \(i = j + l\) for some \(l \in \{1, \ldots, d\}\). Let \(\bar{\phi}\) denote the induced endomorphism on \(N / H\). Since \(\sqrt[N / H]{\gamma_{d - l + 1}(N / H)}\) is fully characteristic,
	\[
		p(\phi(N_{j + l})) = \bar{\phi}(p(N_{j + l})) = \bar{\phi}\left(\sqrt[N / H]{\gamma_{d - l + 1}(N / H)}\right) \leq \sqrt[N / H]{\gamma_{d - l + 1}(N / H)}.
	\]

	This implies that \(\phi(N_{i}) \leq \inv{p}\left(\sqrt[N / H]{\gamma_{d - l + 1}(N / H)}\right) = N_{i}\).
\end{proof}

\begin{prop}	\label{prop:productFormulaOneSubgroupFGTFNilpotentGroups}
	Let \(N\) be a \fgtf nilpotent group. Let \(H\) be a normal subgroup such that \(N / H\) is torsion-free and suppose that \(\phi \in \End(N)\) is such that \(\phi(H) \leq H\). Let \(\restr{\phi}{H}\) and \(\bar{\phi}\) denote the induced endomorphisms on \(H\) and \(N / H\), respectively. Then \(R(\phi) = R(\restr{\phi}{H})R(\bar{\phi})\).
	
	In particular, if \(H\) is characteristic and has the \(\Rinf\)-property, then \(N\) has the \(\Rinf\)-property as well.
\end{prop}
\begin{proof}
	By \cref{prop:NormalSubgroupInFGTFNilpotentGroupContainedInCentralSeries}, there is a \(\phi\)-ITF-series \(1 = H_{0} \lhd H_{1} \lhd \ldots \lhd H_{c - 1} \lhd H_{c} = N\) containing \(H\) as one of its groups, say, \(H = H_{j}\). The product formula yields
	\begin{equation}	\label{eq:productFormulaSplitUp}
		R(\phi) = \prod_{i = 1}^{j} R(\phi_{i}) \prod_{i = j + 1}^{c} R(\phi_{i}),
	\end{equation}
	where \(\phi_{i}\) is the induced endomorphism on \(H_{i} / H_{i - 1}\). As \(1 \lhd H_{1} \lhd \ldots H_{i - 1} \lhd H_{i} = H\) is a \(\restr{\phi}{H}\)-ITF-series of \(H\), the first product in \eqref{eq:productFormulaSplitUp} equals \(R(\restr{\phi}{H})\) by \cref{theo:productFormulaFGTFNilpotentGroups}. Analogously as in the proof of \cref{theo:productFormulaFGTFNilpotentGroups}, one can argue that
	\[
		\prod_{i = j + 1}^{c} R(\phi_{i}) = R(\bar{\phi}).
	\]
	Therefore, the result follows.
	
	For the claim about the \(\Rinf\)-property, note that every automorphism of \(N\) restricts to one of \(H\) if the latter is characteristic. So, \(R(\restr{\phi}{H}) = \infty\) for every \(\phi \in \Aut(N)\). Consequently, \(R(\phi) = \infty\) for every \(\phi \in \Aut(N)\).
\end{proof}

We can drop three conditions and still show that \(\phi \in \End(N)\) has infinite Reidemeister number if its restriction to \(H\) does.
We start by dropping the condition that \(N / H\) is torsion-free.

\begin{prop}	\label{prop:InfiniteReidemeisterNumberOnNormalSubgroupImpliesInfiniteReidemeisterNumber}
	Let \(N\) be a \fgtf nilpotent group, \(H\) a normal subgroup, and \(\phi \in \End(N)\) such that \(\phi(H) \leq H\). Suppose that \(R(\restr{\phi}{H}) = \infty\). Then \(R(\phi) = \infty\) as well.
\end{prop}
\begin{proof}
	We aim to apply \cref{prop:productFormulaOneSubgroupFGTFNilpotentGroups}. However, the quotient \(N / H\) is not necessarily torsion-free. To amend this, we will use \(M := \sqrt[N]{H}\). By \cref{prop:IsolatorOfSubgroupIsSubgroupNilpotentGroups}, \(M\) is a subgroup and by \cref{lem:IsolatorOfPhiInvariantIsPhiInvariant}, \(\phi(M) \leq M\). We argue that \(M\) is normal and that \(N / M\) is torsion-free.
	
	First, let \(n \in N\) and \(m \in M\). Then there is a \(k \geq 1\) such that \(m^{k} \in H\). Since \(H\) is normal, \((\inv{n} mn)^k = \inv{n} m^{k} n \in H\). Therefore, \(\inv{n}mn \in M\), which shows that \(M\) is normal.

Next, suppose that \(n \in N\) is such that \((nM)^{d} = M\) for some \(d \geq 1\). Then \(n^{d} \in M\), which means that \((n^{d})^{k} \in H\) for some \(k \geq 1\). Consequently, \(n^{dk} \in H\), which implies that \(n \in M\). Hence, \(nM = M\); in other words, \(N / M\) is torsion-free.

Finally, we show that \(R(\phi) = \infty\). The induced automorphism \(\restr{\phi}{H}\) has infinite Reidemeister number, by assumption, and \(H\) has finite index in \(M\) by \cref{prop:IsolatorOfSubgroupIsSubgroupNilpotentGroups}. By \cite[Lemma~1.2(4)]{GoncalvesWong06}, the combination of both implies that \(R(\restr{\phi}{M}) = \infty\) as well. It then follows from \cref{prop:productFormulaOneSubgroupFGTFNilpotentGroups} that \(R(\phi) = R(\restr{\phi}{M}) R(\bar{\phi})\), where \(\bar{\phi}\) is the induced automorphism on \(N / M\). Therefore, \(R(\phi) = \infty\).
\end{proof}

Next, we drop the normality condition.

\begin{prop}	\label{prop:InfiniteReidemeisterNumberOnSubgroupFGTFNilpotentGroup}
	Let \(N\) be a \fgtf nilpotent group and \(\phi \in \End(N)\). Suppose that \(H\) is a subgroup such that \(\phi(H) \leq H\) and such that \(R(\restr{\phi}{H}) = \infty\). Then \(R(\phi) = \infty\) as well.
\end{prop}

\begin{proof}
	Let \(c\) be the nilpotency class of \(N\) and write \(C_{i} = \gamma_{c - i + 1}(N)\) for each \(i \in \{0, \ldots, c\}\).
	Define the subgroup \(H_{i} := HC_{i}\) for each \(i \in \{0, \ldots, c\}\).
	We first argue that \(H_{i}\) is normal in \(H_{i + 1}\) for each \(i \in \{0, \ldots, c - 1\}\).
		
	Let \(i \in \{0, \ldots, i\}\), \(h_{i} \in H_{i}\) and \(h_{i + 1} \in H_{i + 1}\) be arbitrary.
	We prove that \([h_{i}, h_{i + 1}] \in H_{i}\).
	Write \(h_{i} = h z\) and \(h_{i + 1} = h' z'\) for some \(h, h' \in H\), \(z \in C_{i}\) and \(z' \in C_{i + 1}\).
	Then
	\begin{align*}
		[h_{i}, h_{i + 1}]	&=	[hz, h'z']	\\
					&=	[hz, z'][hz, h']^{z'}	\\
					&= 	[hz, z'][h, h']^{zz'} [z, h']^{z'}	\\
					&=	[hz, z'][h, h'] [[h, h'],zz'] [z, h']^{z'}	\\
					&\in 	[N, C_{i + 1}] H [H, C_{i + 1}] [C_{i}, N]^{N}			\\
					& \subseteq C_{i} H C_{i} C_{i}	\\
					& \subseteq HC_{i} = H_{i}
	\end{align*}
	since \([C_{k}, N] \leq C_{k - 1} \lhd N\) for all \(k \geq 1\).
	
	So, we have the subnormal series
	\[
		H = H_{0} \lhd H_{1} \lhd \ldots \lhd H_{c} = N.
	\]
	Since \(\phi(H) \leq H\) by assumption and \(\phi(\gamma_{k}(N)) \leq \gamma_{k}(N)\) for all \(k \geq 1\), also \(\phi(H_{i}) \leq H_{i}\) for each \(i \in \{0, \ldots, c\}\).
	Therefore, for \(i \in \{0, \ldots, c\}\), let \(\phi_{i}\) be the induced endomorphism on \(H_{i}\).

	Since \(R(\phi_{0}) = R(\restr{\phi}{H}) = \infty\), \cref{prop:InfiniteReidemeisterNumberOnNormalSubgroupImpliesInfiniteReidemeisterNumber} implies that \(R(\phi_{1}) = \infty\) as well, as \(H_{0}\) is normal in \(H_{1}\). By induction, \(R(\phi_{i}) = \infty\) for each \(i \geq 0\). In particular, \(R(\phi) = R(\phi_{c}) = \infty\).
\end{proof}

Finally, we drop the condition that \(N\) is torsion-free.
\begin{theorem}	\label{theo:InfiniteReidemeisterNumberOnSubgroupFGNilpotentGroup}
	Let \(N\) be a \fg nilpotent group and \(\phi \in \End(N)\). Suppose that \(H\) is a subgroup such that \(\phi(H) \leq H\) and such that \(R(\restr{\phi}{H}) = \infty\). Then \(R(\phi) = \infty\) as well.
\end{theorem}
\begin{proof}
	Let \(T := \tau(N)\) be the torsion-subgroup of \(N\).
	Since \(\phi(T) \leq T\), \(\phi\) induces an endomorphism \(\bar{\phi}\) on \(N / T\).
	We prove that \(R(\bar{\phi}) = \infty\), for then \cref{lem:LowerBoundReidemeisterNumberByQuotient} implies that \(R(\phi) = \infty\) as well.
	
	Consider the subgroup \(HT / T\) of \(N / T\), which satisfies \(\bar{\phi}(HT / T) \leq HT / T\) by assumption on \(H\).
	Let \(\restr{\bar{\phi}}{HT / T}\) denote the induced endomorphism on \(HT / T\).
	We claim that the map
	\[
		f: \Reid[\restr{\phi}{H}] \to \Reid[\restr{\bar{\phi}}{HT / T}]: [h]_{\phi} \mapsto [hT]_{\restr{\bar{\phi}}{HT / T}}
	\]
	is a well-defined surjection with finite fibres.
	
	Suppose that \(h = g h' \inv{\phi}(g)\) for some \(h, h', g \in H\). Projecting down to \(N / T\), we get
	\[
		hT = gh' \inv{\bar{\phi}(g)}T = gh'\inv{\restr{\bar{\phi}}{HT / T}(g)}T.
	\]
	This proves that \([hT]_{\restr{\bar{\phi}}{HT / T}} = [h'T]_{\restr{\bar{\phi}}{HT / T}}\).
	Since each element in \(HT / T\) is of the form \(hT\) for some \(h \in H\), \(f\) is also surjective.
	
	Lastly, to prove that \(f\) has finite fibres, let \(h_{0} \in H\) be fixed and suppose that \([hT]_{\restr{\bar{\phi}}{HT / T}} = [h_{0}T]_{\restr{\bar{\phi}}{HT / T}}\) for some \(h \in H\).
	Then there is some \(g \in H\) such that
	\[
		hT = gh_{0}\inv{\restr{\bar{\phi}}{HT / T}(g)}T.
	\]
	In \(N\), this means that there is some \(t \in T\) such that
	\[
		h = gh_{0}t \inv{\restr{\phi}{H}(g)}.
	\]
	In other words, \([h]_{\restr{\phi}{H}} = [h_{0}t]_{\restr{\phi}{H}}\) for some \(t \in T\).
	Thus,
	\[
		\inv{f}([h_{0}T]_{\restr{\bar{\phi}}{HT / T}}) \subseteq \{[h_{0}t]_{\restr{\phi}{H}} \mid t \in T\}.
	\]
	Since \(N\) is \fg, \(T\) is finite.
	Hence, the set on the right-hand side above is also finite, which proves that \(f\) has finite fibres.
	As \(\Reid[\restr{\phi}{H}]\) is infinite, this implies that the image of \(f\), namely \(\Reid[\restr{\bar{\phi}}{HT / T}]\), is also infinite.
	
	Thus, \(R(\restr{\bar{\phi}}{HT / T}) = \infty\).
	Now, the group \(N / T\) is \fgtf nilpotent and \(HT / T\) is a \(\phi\)-invariant subgroup with \(R(\restr{\bar{\phi}}{HT / T}) = \infty\).
	\cref{prop:InfiniteReidemeisterNumberOnSubgroupFGTFNilpotentGroup} then implies that \(R(\bar{\phi}) = \infty\) as well.
	This finishes the proof.
\end{proof}
Note that this is a particularly strong tool to `lift' an infinite Reidemeister number from a subgroup to the whole group:
the only condition we impose on \(H\) is that it is \(\phi\)-invariant, which is practically unavoidable to speak of the Reidemeister number of \(\restr{\phi}{H}\) on \(H\).

\subsection{Application 2: twisted stabilisers in nilpotent groups}

Since \fgtf nilpotent groups are residually finite and of type \((F)\), \cref{prop:infiniteTwistedStabilisersImpliesInfiniteReidemeisterNumber} holds for them. The converse holds as well for them, in an even stronger form, which is the content of \cref{prop:InfiniteReidemeisterNumberImpliesInfiniteStabilisersFGTFNilpotent}.

\begin{lemma}	\label{lem:LiftingFixedPointsCentralQuotient}
	Let \(N\) be a \fg nilpotent group and let \(C\) be a central subgroup such that \(N / C\) is \tf. Let \(\phi \in \End(N)\) be such that \(\phi(C) \leq C\) and let \(\bar{\phi}\) denote the induced endomorphism on \(N / C\). If \(\bar{\phi}\) has a non-trivial fixed point, then so does \(\phi\).
\end{lemma}
\begin{proof}
	Suppose \(\bar{\phi}(nC) = nC\) for some \(n \notin C\). Let \(\restr{\phi}{C}\) denote the restriction of \(\phi\) to \(C\). If \(R(\restr{\phi}{C}) = \infty\), then \(\restr{\phi}{C}\) has a non-trivial fixed point by \cref{prop:ReidemeisterNumbersFreeAbelianGroups}. Hence, in that case, \(\phi\) has one as well.
	
	So, suppose that \(R(\restr{\phi}{C}) < \infty\). Since \(\bar{\phi}(nC) = nC\), there is a \(c \in C\) such that \(\phi(n) = nc\). As \(C\) is central, \(\phi(n^{m}) = n^{m} c^{m}\) for all \(m \geq 1\). Since \(\Reid[\restr{\phi}{C}]\) has a group structure by \cref{prop:GroupStructureReidemeisterClassesAbelianGroups} and \(R(\restr{\phi}{C}) < \infty\), it is a finite group. Thus, there is a \(k \geq 1\) such that \(c^{k} \Rconj{\restr{\phi}{C}} 1\), say, \(c^{k} = c_{k} \inv{\phi(c_{k})}\) for some \(c_{k} \in C\). We then get
	\[
		c_{k} \inv{\phi(c_{k})} = c^{k} = n^{-k} \phi(n^{k}),
	\]
	from which we derive that \(\phi(n^{k}c_{k}) = n^{k} c_{k}\). If \(n^{k} c_{k} \in C\), then
	\[
		(nC)^{k} = n^{k}C = n^{k}c_{k}C = C.
	\]
	Since \(N / C\) is \tf, it follows that \(nC = C\), which contradicts the initial assumption on \(n\). Therefore, \(n^{k}c_{k}\) does not lie in \(C\), which in particular shows that \(n^{k} c_{k} \ne 1\). Therefore, \(\phi\) has a non-trivial fixed point.
\end{proof}
\begin{prop}	\label{prop:InfiniteReidemeisterNumberImpliesInfiniteStabilisersFGTFNilpotent}
	Let \(N\) be a \fgtf nilpotent group. Let \(\phi \in \End(N)\) and suppose that \(R(\phi) = \infty\). Then \(\Stab_{\phi}(n)\) is infinite for all \(n \in N\).
\end{prop}
\begin{proof}
	Fix \(\phi \in \End(N)\) with \(R(\phi) = \infty\). We start by proving that \(\phi\) has a non-trivial fixed point. Let
	\[
		1 = N_{0} \lhd N_{1} \lhd \ldots \lhd N_{c} = N
	\]
	be the adapted lower central series. Let \(\phi_{i}\) denote the induced endomorphism on \(\frac{N_{i}}{N_{i - 1}}\). By \cref{theo:productFormulaFGTFNilpotentGroups}, there exists an \(i \in \{1, \ldots, c\}\) such that \(R(\phi_{i}) = \infty\). Since \(N_{i} / N_{i - 1}\) is \fgtf abelian, \(\phi_{i}\) has a non-trivial fixed point by \cref{prop:ReidemeisterNumbersFreeAbelianGroups}. It follows that the induced endomorphism on \(N / N_{i - 1}\), denoted by \(\psi_{i}\), has a non-trivial fixed point as well. We prove by induction on \(i\) that we can lift this fixed point to a non-trivial fixed point of \(\phi\).
	
	If \(i = 1\), then \(\psi_{1}\) is essentially \(\phi\), which immediately proves the result.
	
So, suppose that the result holds for \(i\), which means that we can lift a non-trivial fixed point of \(\psi_{i}\) to one of \(\phi\), and assume that \(\psi_{i + 1} \in \End(N / N_{i})\) has a non-trivial fixed point. Consider the short exact sequence
	\[
		1 \to \frac{N_{i}}{N_{i - 1}} \to \frac{N}{N_{i - 1}} \to \frac{N / N_{i - 1}}{N_{i} / N_{i - 1}} \to 1.
	\]
	We have the endomorphism \(\psi_{i}\) on the middle term of this sequence. Let \(\bar{\psi}_{i}\) denote the induced endomorphism on the last (non-trivial) term in the sequence. By the third isomorphism theorem and the definition of \(\bar{\psi}_{i}\) and \(\psi_{i + 1}\), the following square commutes:
	\[
		\begin{tikzcd}
			\frac{N / N_{i - 1}}{N_{i} / N_{i - 1}} \arrow{r}{\bar{\psi}_{i}} \arrow{d}{\cong}	&	\frac{N / N_{i - 1}}{N_{i} / N_{i - 1}} \arrow{d}{\cong}	\\
			\frac{N}{N_{i}}	\arrow{r}{\psi_{i + 1}}	&	\frac{N}{N_{i}}
		\end{tikzcd}
	\]
	Consequently, \(\bar{\psi}_{i}\) has a non-trivial fixed point. Now, remark that \(N_{i} / N_{i - 1}\) is a central subgroup of \(N / N_{i - 1}\) and that \(\frac{N / N_{i - 1}}{N_{i} / N_{i - 1}}\) is torsion-free. Therefore, \cref{lem:LiftingFixedPointsCentralQuotient} implies that \(\psi_{i} \in \End(N / N_{i - 1})\) has a non-trivial fixed point as well. The induction hypothesis now yields a non-trivial fixed point for \(\phi\). As \(N\) is torsion-free, \(\Fix(\phi)\) is infinite.

Finally, let \(n \in N\) be arbitrary. By \cref{lem:PropertiesTwistedConjugacyCompositionWithInnerAutomorphism}, \(\Stab_{\phi}(n) = \Fix(\inn{n} \circ \phi)\) and \(R(\inn{n} \circ \phi) = R(\phi)\). Therefore, we can apply the argument above to \(\inn{n} \circ \phi\) to obtain that \(\Fix(\inn{n} \circ \phi) = \Stab_{\phi}(n)\) is infinite for all \(n \in N\).
\end{proof}

\begin{remark}
	This strong converse implication does not hold in general for \fg residually finite groups. Let \(n \geq 2\) be an integer and let \(F_{n}\) be the free group of rank \(n\) with generators \(x_{1}, \ldots, x_{n}\). S.\ Gersten proves that the automorphism of \(F_{n}\) given by
	\[
		\phi(x_{i}) = x_{i + 1}x_{i}, i \in \{1, \ldots, n - 1\}	\quad \text{and} \quad \phi(x_{n}) = x_{1}
	\]
	has trivial fixed-point subgroup \cite[\S 6]{Gersten84}. On the other hand, \(F_{n}\) has the \(\Rinf\)-property \cite{DekimpeGoncalves14, Felshtyn01}, so \(R(\phi) = \infty\).
	
	Another examples involves \(K := \grppres{a, b}{ba\inv{b} = \inv{a}}\), the fundamental group of the Klein bottle.
	The map \(\phi: K \to K\) defined by \(\phi(a) = \inv{a}\) and \(\phi(b) = \inv{b}\) is an automorphism.
	As \(K\) has the \(\Rinf\)-property by \cite[Theorem~2.2]{GoncalvesWong09}, \(R(\phi) = \infty\).
	However, \(\Fix(\phi)\) is trivial.
	Indeed, consider the subgroup \(A := \grpgen{a}\).
	This subgroup is normal and \(\phi\)-invariant, and both \(A\) and \(K / A\) are isomorphic with \(\Z\). 
	Moreover, the induced maps \(\restr{\phi}{A}\) and \(\bar{\phi}\) on \(A\), respectively \(K / A\), are both given by the inversion map.
	Hence, they only have trivial fixed points.
	Consequently, \(\phi\) only has trivial fixed points as well.
	
	There is, however, a twisted stabiliser of \(\phi\) that is infinite.
	Remark that \((\inn{b} \circ \phi)(a) = b \inv{a} \inv{b} = a\).
	Therefore, \(\Stab_{\phi}(a) = \Fix(\inn{b} \circ \phi)\) contains \(\grpgen{a}\), which is infinite.
\end{remark}
These examples raise the following question:
\begin{quest}
	Let \(G\) be a residually finite group of type \((F)\). Suppose that \(\phi \in \End(G)\) has infinite Reidemeister number. Is there necessarily a \(g \in G\) such that \(\Stab_{\phi}(g)\) is infinite?
\end{quest}

\subsection{Application 3: Reidemeister numbers on finite index subgroups}
In general, Reidemeister numbers do not behave well under taking finite index subgroups, even on \fgtf groups.
Consider again \(K = \grppres{a, b}{ba\inv{b} = \inv{a}}\), the fundamental group of the Klein bottle.
Its centre is given by \(\grpgen{a, b^{2}}\), has finite index in \(K\), and is isomorphic with \(\Z^{2}\).
Let \(\phi: K \to K\) be again the automorphism determined by \(\phi(a) = \inv{a}\) and \(\phi(b) = \inv{b}\).
The induced automorphism on \(Z(K)\) is equal to \(-\Id_{\Z^{2}}\), which has Reidemeister number equal to \(4\) by \cref{prop:ReidemeisterNumbersFreeAbelianGroups}.
However, \(R(\phi) = \infty\), since \(K\) has the \(\Rinf\)-property as mentioned earlier.

On \fgtf nilpotent groups, on the other hand, Reidemeister numbers behave well under taking finite index subgroups.
We start with recalling the notion of Hirsch length and proving the special case of \fgtf abelian groups.

For a polycyclic group \(G\), we let \(\hirsch(G)\) denote the \emph{Hirsch length of \(G\)}, that is, the number of infinite cyclic factors in a polycyclic series of \(G\).
\begin{lemma}[{\cite[Exercise~8]{Segal83}}]	\label{lem:HirschLengthProperties}
	Let \(G\) be a polycyclic group, \(H\) a subgroup of \(G\), and \(N\) a normal subgroup of \(G\). Then
	\begin{enumerate}[(1)]
		\item \(\hirsch(H) \leq \hirsch(G)\) with equality if and only if \(H\) has finite index in \(G\).
		\item \(\hirsch(G) = \hirsch(N) + \hirsch(G / N)\).
	\end{enumerate}
\end{lemma}
Since \fgtf nilpotent groups are polycyclic, we can use the notion Hirsch length when working with them.

\begin{lemma}	\label{lem:IndexImageFiniteIndexSubgroup}
	Let \(A\) be a \fgtf abelian group and let \(H\) be a finite index subgroup. Let \(\psi \in \End(A)\) be such that \(\psi(H) \leq H\). Then \([A : \psi(A)] = [H : \psi(H)]\).
\end{lemma}
\begin{proof}
	Since \([A : H]\) is finite, we can write
	\begin{equation}	\label{eq:indexFormula[HpsiH]}
		[H : \psi(H)] = \frac{[A : \psi(H)]}{[A : H]} = \frac{[A : \psi(A)][\psi(A) : \psi(H)]}{[A : H]}.
	\end{equation}
	It is well known that \([\psi(A) : \psi(H)] \leq [A : H]\). Therefore, \([H : \psi(H)] = \infty\) if and only if \([A : \psi(A)] = \infty\), which proves they are equal in that case.
	
	So, suppose that both \([A : \psi(A)]\) and \([H : \psi(H)]\) are finite. This means in particular that \(\psi\) is injective. Indeed, we have the exact sequence
	\[
		0 \to \ker \psi \to A \to \im \psi \to 0
	\]
	of finitely generated abelian groups. Since \(\psi(A)\) has finite index in \(A\), the Hirsch length \(\hirsch(\im \psi)\) of \(\im \psi\) equals that of \(A\) by \cref{lem:HirschLengthProperties}. It follows from the same lemma that \(\hirsch(\ker \psi) = \hirsch(A) - \hirsch(\im \psi) = 0\). Thus, \(\ker \psi\) is a finite subgroup of a torsion-free group, which implies that \(\ker \psi = 0\). For injective homomorphisms, equality holds in \([\psi(A) : \psi(H)] \leq [A : H]\). Therefore, \eqref{eq:indexFormula[HpsiH]} reduces to \([H : \psi(H)] = [A : \psi(A)]\). 
\end{proof}

\begin{prop}	\label{prop:ReidemeisterNumberFiniteIndexSubgroupFreeAbelian}
	Let \(A\) be a \fgtf abelian group and let \(H\) be a finite index subgroup. Suppose that \(\phi \in \End(A)\) is such that \(\phi(H) \leq H\). Let \(\restr{\phi}{H}\) denote its restriction to \(H\). Then \(R(\restr{\phi}{H}) = R(\phi)\).
\end{prop}	
\begin{proof}
	We know by \cref{prop:ReidemeisterNumbersFreeAbelianGroups} that \(R(\phi) = [A : \im(\phi - \Id_{A})]\) and \(R(\restr{\phi}{H}) = [H : \im(\restr{\phi}{H} - \Id_{H})]\). Therefore, applying \cref{lem:IndexImageFiniteIndexSubgroup} to \(\phi - \Id\), we conclude that \({[H : \im(\restr{\phi}{H} - \Id_{H})]} = [A : \im(\phi - \Id_{A})]\).
\end{proof}
\begin{theorem}	\label{theo:ReidemeisterNumberEndomorphismFiniteIndexNormalSubgroupFGTFNilpotentSubgroupEqualsReidemeisterNumber}
	Let \(N\) be a \fgtf nilpotent group and let \(H\) be a finite index subgroup. Let \(\phi \in \End(N)\) be such that \(\phi(H) \leq H\). Let \(\restr{\phi}{H}\) denote its restriction to \(H\). Then \(R(\restr{\phi}{H}) = R(\phi)\).
\end{theorem}
\begin{proof}
	Let \(1 = N_{0} \lhd N_{1} \lhd \ldots \lhd N_{c - 1} \lhd N_{c} = N\) be a \(\phi\)-ITF-series. For each \(i \in \{0, \ldots, c\}\), define \(H_{i} := H \cap N_{i}\). We claim that the \(H_{i}\) form a \(\restr{\phi}{H}\)-ITF-series for \(H\).
	
	First, let \(i \in \{0, \ldots, c\}\) be arbitrary. Since \(\phi(N_{i}) \leq N_{i}\) and \(\phi(H) \leq H\), we find that
	\[
		\phi(H_{i}) = \phi(N_{i} \cap H) \leq \phi(N_{i}) \cap \phi(H) \leq N_{i} \cap H = H_{i}.
	\]
	
	Next, we prove that the \(H_{i}\) form a central series. Fix \(i \in \{1, \ldots, c - 1\}\). Clearly, \([H_{i}, H] \leq H\). The fact that the \(N_{i}\) form a central series implies that
	\[
		[H_{i}, H] \leq [N_{i}, N] \leq N_{i - 1}.
	\]
	Both inclusions together yield \([H_{i}, H] \leq H_{i - 1}\). This proves that the \(H_{i}\) form a central series.
	
	Finally, fix again \(i \in \{1, \ldots, c\}\). Since \(N_{i - 1} \leq N_{i}\), we know that
	\[
		H_{i - 1} = H \cap N_{i - 1} = H \cap N_{i} \cap N_{i - 1} = H_{i} \cap N_{i - 1}.
	\]
	Thus, the second isomorphism theorem applied to \(N_{i - 1}\) and \(H_{i}\) as subgroups of \(N_{i}\) yields
	\[
		\frac{H_{i}}{H_{i - 1}} = \frac{H_{i}}{H_{i} \cap N_{i - 1}} \cong \frac{H_{i} N_{i - 1}}{N_{i - 1}} \leq \frac{N_{i}}{N_{i - 1}}.
	\]
	It follows that \(H_{i} / H_{i - 1}\) is torsion-free, because \(N_{i} / N_{i - 1}\) is.
	
	We conclude that \(1 = H_{0} \lhd H_{1} \ldots \lhd H_{c - 1} \lhd H_{c} = H\) is a \(\restr{\phi}{H}\)-ITF-series for \(H\). Therefore, if we let \(\phi_{i}\) denote the induced endomorphism on \(N_{i} / N_{i - 1}\) and \(\phi_{H, i}\) the one on \(H_{i} / H_{i - 1}\) for each \(i \in \{1, \ldots, c\}\), then the product formula yields
	\begin{equation}	\label{eq:ProductFormulaReidemeisterNumbersNandH}
		R(\phi) = \prod_{i = 1}^{c} R(\phi_{i}) \quad \text{ and } \quad R(\restr{\phi}{H}) = \prod_{i = 1}^{c} R(\phi_{H, i}).
	\end{equation}
	We will prove that \(R(\phi_{i}) = R(\phi_{H, i})\) for each \(i \in \{1, \ldots, c\}\), which then implies the result.
	
Let \(i \in \{1, \ldots, c\}\). Let \(f\) be the canonical isomorphism between \(\frac{H_{i}}{H_{i - 1}}\) and \(\frac{H_{i} N_{i - 1}}{N_{i - 1}} \leq \frac{N_{i}}{N_{i - 1}}\). There are two ways to define an endomorphism on \(\frac{H_{i}N_{i - 1}}{N_{i - 1}}\) using \(\phi\): we can either restrict \(\phi_{i}\) to \(\frac{H_{i} N_{i - 1}}{N_{i - 1}}\), since \(\phi(H_{i}) \leq H_{i}\), which yields a map \(\tilde{\phi}_{i}\); or we can use \(f\) and define \(\psi_{i} := f \circ \phi_{H, i} \circ \inv{f}\). This yields the following diagram:
	\[
	\begin{tikzcd}
		\frac{H_i}{H_{i - 1}} \arrow{r}{f} \arrow{d}{\phi_{H, i}} & \frac{H_{i}N_{i - 1}}{N_{i - 1}} \arrow[r, hook] \arrow[shift right,swap]{d}{\psi_{i}} \arrow[shift left]{d}{\tilde{\phi}_{i}}& \frac{N_{i}}{N_{i - 1}} \arrow{d}{\phi_{i}} \\
		\frac{H_{i}}{H_{i - 1}} \arrow{r}{f}                         & \frac{H_{i}N_{i - 1}}{N_{i - 1}} \arrow[r, hook]                    & \frac{N_{i}}{N_{i - 1}}                      
	\end{tikzcd}	
	\]
	We claim that \(\psi_{i} = \tilde{\phi}_{i}\). Indeed, let \(h \in H_{i}\). Then
	\begin{align*}
		\tilde{\phi}_{i}(hN_{i - 1})	&=	\phi_{i}(hN_{i - 1})	\\
							&=	\phi(h)N_{i - 1}	\\
							&=	f(\phi(h)H_{i - 1})	\\
							&=	f(\phi_{H, i}(hH_{i - 1}))	\\
							&=	f(\phi_{H, i}(\inv{f}(hN_{i - 1})))	\\
							&=	\psi_{i}(hN_{i - 1}),
	\end{align*}
	which shows that \(\psi_{i} = \tilde{\phi}_{i}\). Therefore, their Reidemeister numbers are equal and since \(f\) is an isomorphism,
	\[
		R(\phi_{H, i}) = R(\psi_{i}) = R(\tilde{\phi}_{i})
	\]
	as well. Now, \(N_{i} / N_{i - 1}\) is a \fgtf abelian group and \(\frac{H_{i} N_{i - 1}}{N_{i - 1}}\) has finite index in \(\frac{N_{i}}{N_{i - 1}}\), as
	\[
		\left[\frac{N_{i}}{N_{i - 1}} : \frac{H_{i}N_{i - 1}}{N_{i - 1}}\right] = [N_{i} : H_{i} N_{i -1}] \leq [N_{i} : H_{i}] = [N \cap N_{i} : H \cap N_{i}] \leq [N : H] < \infty
	\]
	
	Consequently, \cref{prop:ReidemeisterNumberFiniteIndexSubgroupFreeAbelian} implies that \(R(\phi_{i}) = R(\tilde{\phi}_{i})\). Thus, combining this with \(R(\phi_{H, i}) = R(\tilde{\phi}_{i})\) and with \eqref{eq:ProductFormulaReidemeisterNumbersNandH}, we find that
	\begin{align*}
		R(\restr{\phi}{H})	=	\prod_{i = 1}^{c} R(\phi_{H, i})
						=	\prod_{i = 1}^{c} R(\phi_{i})	
						=	R(\phi),
	\end{align*}
	which proves the theorem.
\end{proof}

For an arbitrary group \(G\), if \(N\) is a characteristic subgroup of finite index with the \(\Rinf\)-property, then \(G\) has the \(\Rinf\)-property as well \cite[Lemma~2.2(ii)]{MubeenaSankaran14}. For \fgtf nilpotent groups, we can prove a stronger version of this.

\begin{cor}	\label{cor:SpecRFiniteIndexCharacteristicSubgroupFGTFNilpotentGroup}
	Let \(N\) be a \fgtf nilpotent group and \(H\) a characteristic subgroup of finite index. Then \(\SpecR(N) \subseteq \SpecR(H)\).
	
	In particular, if \(N\) has full Reidemeister spectrum, then so does \(H\).
\end{cor}
\begin{proof}
	Since \(H\) is characteristic in \(N\), each automorphism \(\phi\) of \(N\) restricts to one of \(H\), say, \(\restr{\phi}{H}\). As \(R(\phi) = R(\restr{\phi}{H})\) by \cref{theo:ReidemeisterNumberEndomorphismFiniteIndexNormalSubgroupFGTFNilpotentSubgroupEqualsReidemeisterNumber}, each element of \(\SpecR(N)\) is contained in \(\SpecR(H)\).
\end{proof}

There are only a few examples known of groups with full Reidemeister spectrum: the groups \(\Z^{n}\) with \(n \geq 2\), the free nilpotent groups \(N_{r, 2}\) of rank \(r\) and class \(2\) with \(r \geq 4\) \cite[\S 4]{DekimpeTertooyVargas20} and the free group of countable infinite rank \cite[\S 5]{DekimpeGoncalves14}. With this corollary, we can find new examples of groups with full Reidemeister spectrum.

For integers \(r \geq 4\) and \(m \geq 1\), we define \(N_{r, 2}(m) := \grpgen{x^{m} \mid x \in N_{r, 2}}\). By construction, for each \(x \in N_{r, 2}\), we have \(x^{m} \in N_{r, 2}(m)\). Then \cref{lem:ConditionFiniteIndexSubgroupFGNilpotentGroups} implies that \(N_{r, 2}(m)\) has finite index in \(N_{r, 2}\). As it is a verbal subgroup, it is characteristic. Since \(\SpecR(N_{r, 2})\) is full,  \cref{cor:SpecRFiniteIndexCharacteristicSubgroupFGTFNilpotentGroup} implies that \(\SpecR(N_{r, 2}(m))\) is full as well.

\begin{prop}
	Let \(r, s \geq 4\) and \(m, n \geq 1\) be integers. Then \(N_{r, 2}(m)\) and \(N_{s, 2}(n)\) are isomorphic if and only if \(r = s\) and \(m = n\).
\end{prop}
\begin{proof}	
	We first look at the Hirsch length of \(N_{r, 2}(m)\). As \(N_{r, 2}(m)\) has finite index in \(N_{r, 2}\), their Hirsch lengths coincide by \cref{lem:HirschLengthProperties}. Since the Hirsch length is additive over exact sequences by the same lemma, we can compute \(\hirsch(N_{r, 2})\) using Witt's formula \cite{Witt37} to find
	\begin{align*}
		\hirsch(N_{r, 2}(m))	&= \hirsch(N_{r, 2})	\\
						&= \hirsch(N_{r, 2} / \gamma_{2}(N_{r, 2})) + \hirsch(\gamma_{2}(N_{r, 2}))	\\
						&= r + \frac{r^{2} - r}{2}	\\
						&= \frac{r^{2} + r}{2}.
	\end{align*}
	Since \(r^{2} + r = s^{2} + s\) if and only if \((r - s)(r + s + 1) = 0\), the groups \(N_{r, 2}(m)\) and \(N_{s, 2}(n)\) are non-isomorphic if \(r \ne s\).
	
	Next, consider the abelianisation \(A(r, m) := N_{r, 2}(m) / \gamma_{2}(N_{r, 2}(m))\). We argue that the order of any torsion element in \(A(r, m)\) divides \(m\) and that \(A(r, m)\) contains an element of order \(m\). This implies that \(N_{r, 2}(m)\) and \(N_{s, 2}(n)\) are not isomorphic if \(m \ne n\): without loss of generality, assume \(m < n\). Then \(A(s, n)\) contains an element of order \(n\), whereas \(A(r, m)\) does not, as \(n\) cannot divide \(m\). Since the abelianisations of isomorphic groups are also isomorphic, \(N_{r, 2}(m)\) and \(N_{s, 2}(n)\) cannot be isomorphic.
	
	So, let \(x \in N_{r, 2}(m)\) be such that \(x \gamma_{2}(N_{r, 2}(m))\) has finite order \(d > 1\). This means that \(x^{d} \in \gamma_{2}(N_{r, 2}(m))\). In particular, \(x^{d} \in \gamma_{2}(N_{r, 2})\) as well. As the quotient group \(N_{r, 2} / \gamma_{2}(N_{r, 2})\) is torsion-free, \(x^{d} \gamma_{2}(N_{r, 2}) = \gamma_{2}(N_{r, 2})\) implies that \(x \in \gamma_{2}(N_{r, 2})\).

	Consequently, we can write
	\begin{equation*}	\label{eq:xAsPowersOfCommutators}
		x = \prod_{i < j} [x_{i}, x_{j}]^{b_{i, j}}
	\end{equation*}
	for some \(b_{i, j} \in \Z\) and where \(x_{1}, \ldots, x_{r}\) are free generators of \(N_{r, 2}\).
	Now, any generator of \(\gamma_{2}(N_{r, 2}(m))\) is of the form \([y^{m}, z^{m}]\) for some \(y, z \in N_{r, 2}\).
	Since we work in \(2\)-step nilpotent groups, \([y^{m}, z^{m}] = [y, z]^{m^2}\).
	Consequently, \(\gamma_{2}(N_{r, 2}(m))\) is generated by
	\[
		\{[x_{i}, x_{j}]^{m^2} \mid 1 \leq i < j \leq r\}
	\]
	Futhermore, as \(\gamma_{2}(N_{r, 2})\) is freely generated by \(\{[x_{i}, x_{j}] \mid 1 \leq i < j \leq r\}\), \(\gamma_{2}(N_{r, 2}(m))\) is freely generated by the set above.
	Thus, if \(x^{d} \in \gamma_{2}(N_{r, 2}(m))\), then \(b_{i, j} d \equiv 0 \bmod m^{2}\) for all \(i < j\).
	Since \(d > 1\), we have \(x \notin \gamma_{2}(N_{r, 2}(m))\), so at least one \(b_{i, j}\) is not a multiple of \(m^{2}\).
	Let \(\I\) denote the set of tuples \((i, j)\) for which \(b_{i, j} \not \equiv 0 \bmod m^{2}\).
	
	Note that the only commutators in \(N_{r, 2}(m)\) are of the form \([y, z]^{m}\) for some \(y, z \in N_{r, 2}\).
	Thus, since \(x \in N_{r, 2}(m)\), \(b_{i, j}\) is a multiple of \(m\) for each \((i, j) \in \I\), say \(b_{i, j} = c_{i, j}m\) for some \(c_{i, j} \in \Z\).
	As \(b_{i, j}d \equiv 0 \bmod m^{2}\) for all \((i, j) \in \I\), this implies that \(c_{i, j}d \equiv 0 \bmod m\) for all \((i, j) \in \I\).
	Consequently, \(d\) is divisible by \(d_{i, j} := \frac{m}{\gcd(c_{i, j}, m)}\) for all \((i, j) \in \I\).
	Since \(d\) is the smallest positive integer for which \(x^{d} \in \gamma_{2}(N_{r, 2}(m))\), \(d\) must be the least common multiple of those \(d_{i, j}\).
	This in particular implies that \(d\) divides \(m\).
	
	Finally, the element \([x_{1}, x_{2}]^{m}\) lies in \(N_{r, 2}(m)\) and its projection in \(A(r, m)\) has order \(m\): for this element, the set \(\I\) equals \(\{(1, 2)\}\) and \(c_{1, 2} = 1\). Therefore, our claim is proven.
\end{proof}

The groups \(N_{r, 2}(m)\) thus provide an infinite family of non-isomorphic groups with full Reidemeister spectrum.

\printbibliography
\end{document}